%% file: main.tex
\documentclass[11pt,twoside,reqno]{amsart}
\usepackage[all]{xy}
        \usepackage {amssymb,latexsym,amsthm,amsmath,mathtools,dsfont,mathrsfs}
        \usepackage{enumitem,color}

\usepackage{mathtools}

\usepackage{babel} 

\usepackage{hyperref}
\long\def\symbolfootnote[#1]#2{\begingroup%
\def\thefootnote{\fnsymbol{footnote}}\footnote[#1]{#2}\endgroup}
\newcommand{\Z}{\ensuremath{\mathbb{Z}}}
\newcommand{\ce}{\ensuremath{\mathscr{C}}}

\newcommand{\tr}{\ensuremath{{}^t\!}}

\newcommand{\fq}{\ensuremath{\mathbb{F}_q}}

\newcommand{\ssp}{\mathcal{D}}
\newcommand{\pa}{\mathcal{P}}

\newcommand{\GL}{\textup{GL}}
\newcommand{\SL}{\textup{SL}}
\newcommand{\SP}{\textup{Sp}}
\newcommand{\Or}{\textup{O}}

\newcommand{\F}{\overline{\mathbb{F}_q}}
\newcommand{\U}{\textup{U}}

\makeatletter
\def\imod#1{\allowbreak\mkern10mu({\operator@font mod}\,\,#1)}
\makeatother
\makeatletter
\renewcommand*\env@matrix[1][*\c@MaxMatrixCols c]{%
  \hskip -\arraycolsep
  \let\@ifnextchar\new@ifnextchar
  \array{#1}}
\makeatother
\usepackage[capitalize,nameinlink]{cleveref} 
\usepackage{comment} 
\usepackage{xcolor} 
\hypersetup{
    colorlinks,
    linkcolor={red!80!black},
    citecolor={green!80!black},
    urlcolor={blue!80!black}
}
\newtheorem{theorem}{Theorem}[section]
\newtheorem{lemma}[theorem]{Lemma}
\newtheorem{corollary}[theorem]{Corollary}
\newtheorem{proposition}[theorem]{Proposition}
\newtheorem*{theorem*}{Theorem}
\theoremstyle{definition}
\newtheorem{definition}[theorem]{Definition}

\numberwithin{equation}{section}
\newcommand{\ignore}[1]{}

\newcommand{\mynote}[1]{}
\begin{document}
\setcounter{section}{0}
\title[Accepatbility of classical groups]{Acceptability of classical groups\\ in non-zero characteristic}
\author{Saikat Panja}
\email{panjasaikat300@gmail.com}
\address{Harish-Chandra Research Institute- Main Building, Chhatnag Rd, Jhusi, Uttar Pradesh 211019, India}
\thanks{The author (Panja) is supported by PDF-M fellowship from HRI.}
\date{\today}
\subjclass[2020]{20C33, 20C99}
\keywords{acceptable groups, classical groups, element-conjugacy, global conjugacy}


\begin{abstract}
A group $G$ is called to be \emph{acceptable} (due to M. Larsen) if for any finite group $H$, two element-conjugate homomorphisms are globally conjugate. We answer the acceptability question for general linear, special linear, unitary, symplectic and all orthogonal groups over an algebraically closed field of non-zero characteristics.
\end{abstract}
\maketitle
\input{Introduction.tex}

\input{gen.tex}
\input{conju.tex}
\input{glsl.tex}
\input{uni.tex}
\input{symp.tex}
\input{conc.tex}
\bibliographystyle{siam}

\end{document}

%% file: Introduction.tex
\section{Introduction}\label{sec:intro}
We start with a basic result, coming from an application of the representation theory of finite groups over complex numbers. Let $G$ be a finite group.
Consider two homomorphisms $\phi_1,~\phi_2:G\longrightarrow\GL(n,\mathbb{C})$. Suppose that 
for all $g\in G$ there exists $h(g)\in G$, such that $\phi_1(g)=h(g)^{-1}\phi_2(g)h(g)^{-1}$. 
In this case, these two homomorphisms will be called \emph{element-conjugate homomorphisms}.
Note that here $h:G\longrightarrow G$ is a function. If there exists a choice of constant function $h$, we say that $\phi_1$ and $\phi_2$ are \emph{globally conjugate}. 
Now, note that if the above-mentioned maps $\phi_1,~\phi_2$ are element-conjugate, then they are globally conjugate. Indeed if they are element-wise conjugate, their corresponding character values are the same on each of the conjugacy classes of $G$.
Hence their characters are the same. Thus the two homomorphisms are globally conjugate. Now suppose we change $\GL(n,\mathbb{C})$ by $\SL(n,\mathbb{C})$ or $\textup{SO}(n,\mathbb{C})$. Then it is natural to ask
whether all element-conjugate homomorphisms are globally conjugate. At this point, we define a group $\Gamma$ to be \emph{acceptable} (or \emph{strongly acceptable}) if for all finite groups (resp. group) $G$, two element-conjugate homomorphisms are globally conjugate. 

The general question of the relation
between element-conjugacy and global conjugacy arises in many contexts such as algebra, number theory, and geometry. For example, the multiplicity-one question in the theory of automorphic forms (see \cite{Bl94}). A motivating question, related to multiplicity-one questions in the theory of automorphic forms, is when a compact group can be the common covering space of a pair of non-isometric isospectral manifolds (See \cite{Su85}). Motivated 
by this M. Larsen has dealt with many of the compact or complex simple Lie groups, viewed as real groups, after showing that the question reduces to the semisimple case in \cite{Mi94}. In a subsequent paper \cite{Mi96} he proved that a connected, simply connected, compact Lie group $G$ is acceptable if and only if it has no direct factors of type $B_n (n\geq4), D_n (n\geq 4),~ E_n, \text{ or } F_4$. This carries over to the complex groups as well. Moreover, the group of type $G_2$ is the only exceptional group (simply connected or not) which is acceptable. In the year 2016, Y. Fang, G. Han and B. Sun proved that the following two statements are equivalent:
\begin{enumerate}
    \item For all connected compact Lie groups $H$ and all continuous homomorphisms $\phi,\psi:H\longrightarrow G$, if $\phi(h)$ and $\psi(h)$ are conjugate in $G$ for all $h\in H$, then $\phi$ and $\psi$ are globally conjugate.
    \item The Lie
algebra of $G$ contains no simple ideal of type $D_n (n \geq4),~ E_6, ~E_7, \text{ or } E_8$.
\end{enumerate}
These results have strong applications in many parts of mathematics. For example how to 
determine whether a given representation is distinguished or not in terms of the Langlands 
parameters (see \cite{AnPr13}). Another question being that do the multiplicities of the 
trivial representation of $H$ in finite-dimensional irreducible representations of $G$ 
determine the conjugacy class of $H$, where $H$ is a closed subgroup of a complex reductive 
group $G$ (see \cite{Wa15}).

Consider a prime $p$ and the target group is considered to be one of the general linear groups, special 
linear groups, unitray groups, symplectic groups or the orthogonal groups  when the field is algebraically closed field $\F$ of characteristic $p$. We prove that all these groups are not acceptable. We are further looking into the cases for exceptional cases, which will be carried  out in future work.
\subsection*{Methodology} To work with the above-mentioned groups, we make use of their conjugacy classes. This has been adapted from the work of Green (\cite{Gr55}) and Macdonald (\cite{Ma81}) for the general linear groups and from the work of Wall (\cite{Wa63}) in the cases 
of unitary, symplectic and orthogonal groups. We will briefly mention these results in due course. The idea is to work with unipotent classes, for which we need to know the 
class representative for the same (see \cite{BuGi16}). The work \cite{GoLiO17} of Gonshaw, Liebeck, and O'Brien has been a rescuer to this. After understanding the conjugacy classes, we move to construct particular homomorphisms from particular groups to the target groups. For most of the cases, we have used the group $\Z/p\Z\oplus\Z/p\Z$ to be the source group. Without further delay, we will present the organization of the paper.
\subsection*{Organization of the paper} We start with some of the general results in \cref{sec:gen}. This section contains two very important results, viz. \cref{lem:product} and \cref{lem:centralizer} which have been the lifeline for most of the proofs. In \cref{sec:conju} we briefly mention the groups and their conjugacy classes. Thereafter \cref{sec:gen-spec-lin}, \cref{sec:unitary groups} and \cref{sec:symplectic-groups} are devoted to the main proof of the results concerning general and special linear groups, unitary groups, symplectic and orthogonal groups, respectively. Our main results are \cref{thm:gl-not-acc}, \cref{thm:unitary-not-acc}, \cref{th:symplectic-not-acc}, \cref{thm:OddOrth-not-acceptable}, \cref{thm:Orth-odd-not-acc}. We finish the paper by addressing a few future questions in \cref{sec:conclusion}.
\subsection*{Acknowledgment}
The author would like to thank professor Anupam Singh from IISER Pune for his helpful discussions. A part of this work was finished during his visit IIT-Jodhpur in December 2022. I would like to take this opportunity to thank the institute for its hospitality during the stay.

%% file: gen.tex
\section{General results}\label{sec:gen}
\begin{lemma}[Proposition 1.1, \cite{Mi94}]\label{lem:product}
    Let $G=G_1\times G_2$ for two groups $G_1,~G_2$. Then $G$ is acceptable if and only if $G_1$ and $G_2$ are acceptable.
\end{lemma}
\begin{proof}
    Let $G_1$ and $G_2$ be two acceptable groups. Assume $\phi_1,\phi_2:H\longrightarrow G_1\times G_2$ be two homomorphisms which are element-conjugate. Then the maps $\pi_1\circ\phi_1,\pi_1\circ\phi_2:H\longrightarrow G_1$ are element conjugate, hence globally conjugate, via say $g_1$. Similarly there exists $g_2\in G_2$, via which $\pi_2\circ\phi_1,\pi_2\circ\phi_2:H\longrightarrow G_2$ are globally conjugate. Hence $\phi_1,\phi_2$ are globally conjugate.

    On the other hand considering the monomorphisms $i_1:G_1\longrightarrow G_1\times G_2$ and $i_2: G_2\longrightarrow G_1\times G_2$, we get that if $G_1\times G_2$ is acceptable, then so are $G_1$ and $G_2$.
\end{proof}
\begin{lemma}\label{lem:centralizer}
    Let $G$ be an acceptable group. Then for all $a\in G$ the subgroup $\ce_G(a)$, the centralizer of $a$ is acceptable. 
\end{lemma}
\begin{proof}
    Let $\varphi_{1},\varphi_2:H\longrightarrow \ce_G(a)$ be two element conjugate homorphisms. Construct the following homomorphisms
    \begin{align}
        \widetilde{\varphi_i}:H\times\langle a \rangle \longrightarrow G,\text{ defined as }\widetilde{\varphi_i}(h,a^m)=\varphi(h)a^m.
    \end{align}
    Then the homomorphisms are element-conjugate and hence globally conjugate since $G$ is acceptable. Thus there exists $g\in G$ such that
    \begin{align*}
        g\varphi_1(h)a^mg^{-1}=\varphi_2(g)\text{ for all }m\in\mathbb{Z},h\in H.
    \end{align*}
    On plugging $m=0$ and $h=e_H$, we get that $g\in \ce_G(a)$. Hence we get the result by putting $m=0$.
\end{proof}
\begin{lemma}[See Lemma 2.3 of \cite{Yu22} for Lie group set-up]\label{lem:abelian-centralizer}
    Let $G$ be an acceptable group and $A$ be an abelian subgroup of $G$. Then $\ce_G(A)$ is an acceptable subgroup.
\end{lemma}
\begin{proof}
    Let $\varphi_{1},\varphi_2:H\longrightarrow \ce_G(A)$ be two element conjugate homomorphisms. Construct the following homomorphisms
    \begin{align}
        \widetilde{\varphi_i}:H\times A \longrightarrow G,\text{ defined as }\widetilde{\varphi_i}(h,a)=\varphi(h)a.
    \end{align}
    Then the homomorphisms are element-conjugate and hence globally conjugate since $G$ is acceptable. Thus there exists $g\in G$ such that
    \begin{align*}
        g\varphi_1(h)ag^{-1}=\varphi_2(g)\text{ for all }a\in A,h\in H.
    \end{align*}
    Assuming $h=e_H$, we get that $g\in \ce_G(A)$. Then taking $a=e_A$, we get that $\varphi_1$ and $\varphi_2$ are globally conjugate via $g$.
\end{proof}

%% file: conju.tex
\section{Conjugacy classes in the groups}\label{sec:conju}
In this section, we will discuss the classical groups in brief and present the description of their conjugacy classes. This will be based on the works \cite{Gr55}, \cite{Ma81} and \cite{Wa63}.
\subsection{General linear group} The general linear group $\GL(n,\F)$ is defined to be the set of all $n\times n$ matrices defined over $\F$. The conjugacy class of this group was firstly determined in works \cite{Gr55} and \cite{Ma81} in the case of finite fields. Let $\Phi$ denote the set of all non-constant, monic, irreducible polynomials $f (x)$  with coefficients in $\F$ which is not equal to the polynomial $x$. Let $\Lambda$ denote the set of all partitions $\lambda=(\lambda_1,\lambda_2,\ldots,\lambda_k)$ with $\lambda_i\geq\lambda_{i+1}\geq 0$ and each $\lambda$ are integers.  The conjugacy class of $\GL(n, \F)$ is in one-one correspondence with a set of all functions
$\alpha:\Phi\longrightarrow\Lambda$ which takes the value empty partition on all but finitely many polynomials in $\Phi$,
and satisfies 
\begin{align*}
\sum\limits_{f\in\Phi}|\alpha(f)|\deg f=n.   
\end{align*} 
Hence the conjugacy classes of  $\GL(n,q)$ are determined by combinatorial data consisting of tuples of the form $(f,\lambda_f)$ where $f\in\Phi$. We note that two matrices in $\GL(n,q)$ are conjugate if and only if they are conjugate in $\GL(n,\F)$, which we will use in \autoref{sec:conclusion}.
\subsection{Unitary group}
For the field $\F$, consider an involution $\sigma$ of $\F$. This induces an automorphism on the ring $\mathbb \F [x]$ and the group $\GL(n,\F)$. The \emph{unitary} group is the group of all matrices in $\GL(n,\F)$ which satisfy $^t\overline{A}JA=J$, where $J$ is given by $\begin{pmatrix}
    ~&&&1\\
    &&1&\\
    &\reflectbox{$\ddots$}&&\\
    1&&&
\end{pmatrix}$ (matrix of a hermitian form) and this group will be denoted by $\U(n,\F)$. 
\begin{definition}
The \textit{twisted dual} of a monic degree $r$ polynomial $f(x)\in \F[x]$ satisfying $f(0)\neq 0$, is the polynomial given by $\widetilde{f}(x)=\overline{f(0)}^{-1}x^r\overline{f}(x^{-1})$.
The polynomial $f$ will be called \emph{$\sim$-symmetric} if $f=\widetilde{f}$. A monic polynomial $f(x)\in \F[x]$, will be called to be \textbf{$\sim$-irreducible}
 if and only if it does not have any proper $\sim$-symmetric factor.  We will denote the set of all $\sim$-irreducible monic polynomials by $\widetilde{\Phi}$
\end{definition}
We know that elements $g,h\in\U(n,\F)$ are conjugate in $\U(n,q)$ if and only if they are conjugate in $\GL(n,\F)$ (see \cite{BuGi16}) and the conjugacy classes in $\U(n,\F)$ are parametrized by
sequences $\{(f,\lambda)|f\in\widetilde{\Phi},\lambda\in\Lambda\}$ is the set of irreducible polynomials. The class representative for the conjugacy classes can be found in \cite{ta3}.
\subsection{Symplectic group and Orthogonal group}
For a vector space $V$ of dimension $2n$ over $\F$, consider the unique (up to similarity) non-degenerate alternating bilinear form on $V$ given by 
$$\left<(x_i)_{i=1}^{2n}, (y_j)_{j=1}^{2n} \right> = \sum_{j=1}^{n}x_jy_{2n+1-j}-\sum_{i=0 } ^{n-1}x_{2n-i}y_{i+1}.$$ 
The symplectic group is the subgroup of $\GL(2n,\F)$ consisting of elements preserving this alternating form on $V$. By fixing an appropriate basis, the matrix of the form can be chosen to be $J= \begin{pmatrix} 0 & \Lambda_n\\ -\Lambda_n & 0 \end{pmatrix}$ where 
$\Lambda_n=\begin{pmatrix} 0 & 0 & \cdots & 0 & 1\\
0 & 0 & \cdots & 1 & 0\\ \vdots & \vdots & \reflectbox{$\ddots$} & \vdots & \vdots\\
1 & 0 & \cdots & 0 & 0 \end{pmatrix}$ and 
$$\SP(2n,\F) =\{A \in \GL(2n,\F) \mid \tr AJA=J\}.$$ 

Now consider $V$ to be an $m$-dimensional vector space over the field $\F$. Then there a unique (up to similarity) non-degenerate quadratic forms on $V$. The orthogonal group consists of elements of $\GL(m,\F)$ which preserve a non-degenerate quadratic form $Q$. 
 With respect to an appropriate basis, we will fix the matrices of the symmetric bilinear forms (associated with the quadratic forms) as follows: 
$$J_{o}=\begin{pmatrix} 0 & 0 & \Lambda_n \\ 0 & \alpha & 0\\ \Lambda_n & 0 & 0 
\end{pmatrix},  J_{e}=\begin{pmatrix} 0& \Lambda_n\\ \Lambda_n & 0
\end{pmatrix}$$ 
where $\alpha\in\fq^\times$, and $\Lambda_l =\begin{pmatrix}
0 & 0 & \cdots & 0 & 1\\ 0 & 0 & \cdots & 1 & 0\\
\vdots & \vdots & \reflectbox{$\ddots$} & \vdots & \vdots\\
1 & 0 & \cdots & 0 & 0 \end{pmatrix}$, an $l\times l$ matrix. Then, the orthogonal group in matrix form is 
$$\Or(2m, \F) = \{A \in \GL(2m,\F)\mid \tr{A}J_{e}A=J_{e}\},$$
$$\Or(2m+1,\F)= \{A\in\GL(2m,\F)\mid \tr{A}J_{o}A=J_{o}\}.$$
To describe the conjugacy classes of finite symplectic and orthogonal groups, we will need the concept of \emph{self-reciprocal polynomials, symplectic and orthogonal partitions}, which are defined below. For examples of the same, the readers are suggested to have a look at \cite{BuGi16}. 
\begin{definition}
The \textit{dual} of a monic degree $r$ polynomial $f(x)\in k[x]$ satisfying $f(0)\neq 0$, is the polynomial given by $f^*(x)=f(0)^{-1}x^rf(x^{-1})$. 
The polynomial $f$ will be called \emph{self reciprocal} if $f=f^*$. A monic polynomial $f(x)\in \fq[x]$, will be called to be \textbf{$*-$irreducible}
 if and only if it does not have any proper self-reciprocal factor.
\end{definition}
\begin{definition}
A \textbf{symplectic partition} is a partition of a number $k$, such that the odd parts 
have even multiplicity. The set
of all symplectic partitions will be denoted as $\ssp'_{\SP}$.
\end{definition}
\begin{definition}
An \textbf{orthogonal partition} is a partition of a number $k$, such that all even parts have even multiplicity. The set
of all orthogonal signed partitions will be denoted as $\ssp'_{\Or}$.
\end{definition}
It can be shown that the characteristic polynomial of a symplectic or orthogonal matrix is self-reciprocal. We follow 
J. Milnor's terminology \cite{mi} to distinguish between the $*$-irreducible factors of the characteristic polynomials. We call a $*$-irreducible polynomial $f$ to be

\noindent(1) Type $1$\index{Type $1$ polynomial} if $f=f^*$ and $f$ is irreducible polynomial of even degree;\

\noindent(2) Type $2$\index{Type $2$ polynomial} if $f=gg^*$ and $g$ is irreducible polynomial satisfying $g\neq g^*$;\

\noindent(3) Type $3$\index{Type $3$ polynomial} if $f(x)=x\pm 1$.\

According to \cite{Wa63} the conjugacy classes of $\SP(2n,\F)$ are parameterized by the functions 
$\lambda:\Phi\rightarrow\pa^{2n}\cup\ssp_{\SP}^{2n}$, where $\Phi$ denotes the set of all
 monic,  non-constant, irreducible polynomials, $\pa^{2n}$ is
 the set of all partitions 
of $1\leq k\leq 2n$ and $\ssp_{\SP}^{2n}$ is the set of all symplectic partitions 
of $1\leq k\leq 2n$. Such a $\lambda$ represent a conjugacy class of $\SP$
if and only if \begin{enumerate}
\itemindent=-13pt
\item $\lambda(x)=0$,
\item $\lambda_{\varphi^*}=\lambda_\varphi$,
\item $\lambda_\varphi\in\ssp^n_{\SP}$ iff $\varphi=x\pm 1$,
\item $\displaystyle\sum_{\varphi}|\lambda_\varphi|\textup{deg}(\varphi)=2n$.
\end{enumerate}
Also from \cite{Wa63}, we find out that a similar kind of statement is true for the groups $\Or(n,\F)$. The conjugacy classes of $\Or(n,\F)$ are parameterized by the functions 
$\lambda:\Phi\rightarrow\pa^{n}\cup\ssp_{\Or}^{n}$, where $\Phi$ denotes the set of all
 monic,  non-constant, irreducible polynomials, $\pa^{n}$ is
 the set of all partitions 
of $1\leq k\leq n$ and $\ssp_{\Or}^{n}$ is the set of all orthogonal partitions 
of $1\leq k\leq n$. Such a $\lambda$ represent a conjugacy class of $\Or(n,\F)$
if and only if \begin{enumerate}
\itemindent=-13pt
\item $\lambda(x)=0$,
\item $\lambda_{\varphi^*}=\lambda_\varphi$,
\item $\lambda_\varphi\in\ssp^n_{\Or}$ iff $\varphi=x\pm 1$,
\item $\displaystyle\sum_{\varphi}|\lambda_\varphi|\textup{deg}(\varphi)=n$.
\end{enumerate}
Class representative corresponding to given data can be found in \cite{ta1}, \cite{ta2}, \cite{GoLiO17} 
and we will mention them whenever needed. 

%% file: glsl.tex
\section{General and special linear groups}\label{sec:gen-spec-lin}
\begin{lemma}\label{lem:gl-smaller-case}
For a field $\F$, if $\GL(n+1,\F)$ is acceptable, then so is $\GL(n,\F)$.
\end{lemma}
\begin{proof}
We are going to use \cref{lem:centralizer}. Take the following element
\begin{align*}
    a=\begin{pmatrix}
    I_{n}&0\\
    0&-1
    \end{pmatrix}\in\GL(n+1,\F).
\end{align*}
Then we know that
\begin{align*}
    \ce_{\GL(n+1,\F)}(a)=&\left\{\begin{pmatrix}A&0\\0&\alpha\end{pmatrix}\middle\vert\substack{A\in\GL(n,\F)\\\alpha\in\F^*}\right\}\\
    \cong &\GL(n,\F)\times\F^*.
\end{align*}
Hence using \cref{lem:product}, we conclude that $\GL(n,\F)$ is acceptable.
\end{proof}
\begin{proposition}\label{prop:gl2-not-acc}
For $q\geq 5$ the group $\GL(2,\F)$ is not acceptable.
\end{proposition}
\begin{proof}
    Let $a\neq b$ be two  nonzero elements of $\F$ of order $p$, where $q=p^f$. Further, assume that $a^2\neq b^2$. Note that such a choice exists since $x^2=1$ has less than or equal to two solutions in $\F$. Consider the following two homomorphisms
    \begin{align*}
        &\varphi_1,\varphi_2:\Z/p\Z\oplus\Z/p\Z\longrightarrow\GL(2,\F)\\
        \text{defined as }&\varphi_1(1,0)=\begin{pmatrix}
            1&a\\0&1
        \end{pmatrix}
        ,\varphi_1(0,1)=\begin{pmatrix}
            1&b\\0&1
        \end{pmatrix}\\
        &\varphi_2(1,0)=\begin{pmatrix}
            1&b\\0&1
        \end{pmatrix},
        \varphi_2(0,1)=\begin{pmatrix}
            1&a\\0&1
        \end{pmatrix}.
    \end{align*}
    Then we know that these two homomorphisms are element-conjugate. We now show that these two homomorphisms are not globally conjugate. On the contrary, assume that they are. So there is a matrix $\begin{pmatrix}
        p&q\\r&s
    \end{pmatrix}\in \GL(2,\F)$ acting as a global conjugator. Hence we have the following equations:
    \begin{align*}
        &\begin{pmatrix}
            p&q\\r&s
        \end{pmatrix}
        \begin{pmatrix}
            1&a\\0&1
        \end{pmatrix}
        =\begin{pmatrix}
            1&b\\0&1
        \end{pmatrix}
        \begin{pmatrix}
            p&q\\r&s
        \end{pmatrix}\\
\text{and }&\begin{pmatrix}
            p&q\\r&s
        \end{pmatrix}
        \begin{pmatrix}
            1&b\\0&1
        \end{pmatrix}
        =\begin{pmatrix}
            1&a\\0&1
        \end{pmatrix}
        \begin{pmatrix}
            p&q\\r&s
        \end{pmatrix}.
    \end{align*}
    Solving the polynomial equations obtained by comparing the coefficients, we get that $r=s=0$. This implies that there doesn't exist any matrix playing the role of a global conjugator. Hence the conclusion follows.
\end{proof}
\begin{theorem}\label{thm:gl-not-acc}
Let $p$ be a prime and $q=p^f$. For $q>5$ the group $\GL(n,\F)$ is not acceptable for all $n\in \mathbb{N}$.
\end{theorem}
\begin{proof}
    Follows from \cref{prop:gl2-not-acc} and \cref{lem:gl-smaller-case}.
\end{proof}
\begin{corollary}\label{cor:sl-odd}
The group $\SL(2n+1,\F)$ is not acceptable for all  $n\in\mathbb{N}$.
\end{corollary}
\begin{proof}
    We are going to use \cref{lem:centralizer}. Assume if possible $\SL(2n+1,\F)$ to be acceptable. Choose
    \begin{align*}
        a=\begin{pmatrix}
            -I_{2n}&0\\
            0&1
        \end{pmatrix}\in\SL(2n+1,\F).
    \end{align*}
    Then we get that \begin{align*}
        \ce_{\SL(2n+1,\F)}(a)&=\left\{\begin{pmatrix}A&0\\0&\alpha\end{pmatrix}\middle\vert\substack{A\in\GL(n,\F)\\\alpha\in\F^*\\\det A\cdot\alpha=1}\right\}\\
    \cong &\left\{\begin{pmatrix}A&0\\0&\det A^{-1}\end{pmatrix}\middle\vert A\in\GL(n,\F)\right\}\\
    \cong&\GL(2n,\F).
    \end{align*}
    Hence the result follows from \cref{lem:centralizer} and \cref{thm:gl-not-acc}.
\end{proof}
\begin{proposition}\label{prop:sl-not-accep-power-map}
    Let the power map $\psi_n:\F\longrightarrow\F$ defined as $\psi_n(x)=x^n$ be a surjection. Then $\SL(n,\F)$ is not acceptable.
\end{proposition}
\begin{proof}
    On the contrary assume that $\SL(n,\F)$ is acceptable. We will show that in this case, it will imply $\GL(n,\F)$ to be acceptable. Note that the map $\psi_k$ is a bijection. Let $\varphi_1,\varphi_2:G\longrightarrow\GL(n,\F)$ be two element conjugate homomorphisms. Consider the following two maps
    \begin{align*}
        &\widetilde{\varphi_1},\widetilde{\varphi_2}:G\longrightarrow\SL(n,\F)\\
        \text{defined as }&\widetilde{\varphi_i}(x)=\psi_n^{-1}(\det\varphi_i(x))\varphi_i(x).
    \end{align*}
    Note that $\widetilde{\varphi_1}$ and $\widetilde{\varphi_2}$ are element conjugate homomorphisms. Hence by assumption, they are globally conjugate. Since $\psi_n$ is a bijection and
    conjugate matrices have the same determinant, we get that $\varphi_1$ and $\varphi_2$ are globally conjugate. This contradicts \cref{thm:gl-not-acc} and finishes the proof.
\end{proof}
\begin{corollary}
    Let $(q-1,n)=1$. Then $\SL(n,\F)$ is not acceptable.
\end{corollary}
\begin{proposition}\label{prop:sl4q-not-acc}
    For $n\geq 4$, the group $\SL(n,\F)$ is not acceptable.
\end{proposition}
\begin{proof}
    Fix $\lambda\in\F$ such that $\lambda^2\neq 1$. Fix the element $a=\begin{pmatrix}
        \lambda&&\\
        &\dfrac{1}{\lambda}&\\
        &&I
    \end{pmatrix}\in\SL(n,\F)$.
    Then we have that
    \begin{align*}
        \ce_{\SL(n,\F)}(a)=&\left\{\begin{pmatrix}
            \alpha&&\\
            &\beta&&\\
            &&A
        \end{pmatrix}\middle\vert \substack{\alpha,\beta\in\F^*\\A\in\GL(n-2,\F)\\\alpha\beta\det A=1}\right\}\\
        \cong&\F^*\times\GL(n-2,\F).
    \end{align*}
    Since $\GL(n-2,\F)$ is not acceptable, using \cref{lem:product}, we get that $\SL(n,\F)$ is not acceptable.
\end{proof}

%% file: uni.tex
\section{Unitary groups}\label{sec:unitary groups}
\begin{lemma}\label{lem:unitary-induction}
For a field $\F$, if $\U(n+1,\F)$ is acceptable, then so is $\U(n,\F)$.
\end{lemma}
\begin{proof}
We work with the hermitian form, which is given by the matrix
\begin{align*}
    \begin{pmatrix}
    1&&&\\
    &1&&\\
    &&\ddots&\\
    &&&1
    \end{pmatrix}.
\end{align*}
Consider the element $a=\begin{pmatrix}
    I_n&\\
    &-1
\end{pmatrix}$. Then by \cref{lem:centralizer}, we get that the $\ce_{\U(n+1,\F)}(a)$ is acceptable, if $\U(n+1,\F)$ is. Now we have that
\begin{align*}
    \ce_{\U(n+1,\F)}(a)&=\left\{\begin{pmatrix}
        A&\\
        &\alpha
    \end{pmatrix}\middle\vert\substack{A\in\U(n,\F)\\\alpha\sigma(\alpha)=1}\right\}\\
    &=\U(n,\F)\times\{\alpha\in\F^*|\alpha\sigma(\alpha)=1\}.
\end{align*}
Hence the result follows from \cref{lem:product}.
\end{proof}
\begin{proposition}\label{prop:unit-4}
The group $\U(4,\F)$ is not acceptable.
\end{proposition}
\begin{proof}
    We will work with the hermitian form given by the matrix $
        \begin{pmatrix}
            &&&1\\
            &&1&\\
            &\reflectbox{$\ddots$}&&\\
            1&&&
        \end{pmatrix}$.
    Firstly fix $1\neq a\in\F^*$. Then consider the following two matrices
    \begin{align*}
        X_1=\begin{pmatrix}
            1&1&&\\
            &1&&\\
            &&1&\\
            &&-1&1
        \end{pmatrix}\text{ and }X_a=\begin{pmatrix}
            1&a&&\\
            &1&&\\
            &&1&\\
            &&-a&1
        \end{pmatrix}.
    \end{align*}
    Then we know that these two matrices commute and are conjugate to each other. Consider the following two homomorphisms
    \begin{align*}
        &\varphi_1,\varphi_2:\Z/p\Z\oplus\Z/p\Z\longrightarrow\U(4,\F)\\
        \text{defined as }&\varphi_1((1,0))=X_1,\varphi_1((0,1))=X_a,\\
                          &\varphi_2((1,0))=X_a,\varphi_2((0,1))=X_1.
    \end{align*}
    Then we know that these two homomorphisms are element conjugate but not globally conjugate. Hence $\U(4,\F)$ is not acceptable. 
\end{proof}
\begin{theorem}\label{thm:unitary-not-acc}
For and $n\geq 4$, the groups $\U(n,\F)$ is not acceptable.
\end{theorem}
\begin{proof}
    Follows directly from \cref{lem:unitary-induction} and \cref{prop:unit-4}. 
\end{proof}

%% file: symp.tex
\section{Symplectic and Orthogonal groups}\label{sec:symplectic-groups}
\begin{proposition}\label{prop:smaller-symp}
For a field $\F$ if $\SP(2m+2,\F)$  is an acceptable group, then so is $\SP(2m,\F)$.
\end{proposition}
\begin{proof}
Fix an element $\lambda\in\F^*$ such that $\lambda^2\neq 1$. Then fix the element
\begin{align*}
a=\begin{pmatrix}
\lambda&&&&&\\
&\dfrac{1}{\lambda}&&&&\\
&&&1&&\\
&&-1&&&\\
&&&\ddots&&\\
&&&&&1\\
&&&&-1&
\end{pmatrix} \in\SP(2m+2,\F).   
\end{align*}
Then we know that
\begin{align*}
    \ce_{\SP(2m+2,\F)}(a)&=\ce_{\GL(2m+2,\F)}(a)\cap\SP(2m+2,\F)\\
    &=\left\{\begin{pmatrix}
        \alpha&&\\
        &\beta&\\
        &&A\\
    \end{pmatrix}\middle\vert\substack{\alpha,\beta\in\F^*\\A\in\GL(2m,\F)}\right\}\cap\SP(2m+2,\F)\\
    &\cong\left\{\begin{pmatrix}
        \alpha&\\
        &\frac{1}{\alpha}
    \end{pmatrix}\middle\vert\alpha\in\F^*\right\}\times\SP(2m,\F).
\end{align*}
Hence the result follows from \cref{lem:product}.
\end{proof}
\begin{proposition}\label{prop:sp2-not-acc}
The group $\SP(2,\F)$ is not acceptable.
\end{proposition}
\begin{proof}
    Fix $a\in(\F^*)^2\setminus\{1\}$. It is easy to see that
    \begin{align*}
        \begin{pmatrix}
            1&1\\&1
        \end{pmatrix},\begin{pmatrix}
            1&a\\&1
        \end{pmatrix}\in\SP(2,\F)
    \end{align*}
    are conjugate to each other by an element of $\SP(2,\F)$. Consider the following two homomorphisms 
    \begin{align*}
            &\varphi_1,\varphi_2:\Z/p\Z\oplus\Z/p\Z\longrightarrow\SP(2,\F)\\
            \text{given by }&\varphi_1\left((1,0)\right)=\begin{pmatrix}
                1&1\\&1
            \end{pmatrix},\varphi_1\left((0,1)\right)=\begin{pmatrix}
                1&a\\&1
            \end{pmatrix},\\
            &\varphi_2\left((1,0)\right)=\begin{pmatrix}
                1&a\\&1
            \end{pmatrix},\varphi_2\left((0,1)\right)=\begin{pmatrix}
                1&1\\&1
            \end{pmatrix}.
    \end{align*}
    Then using the same argument as \cref{prop:gl2-not-acc} we get that these two homomorphisms are not globally-conjugate. Hence the result follows.
\end{proof}
\begin{corollary}\label{cor:sp2-not-acc}
The group $\SL(2,\F)$ is not acceptable.
\end{corollary}
\begin{proof}
    Since $\SL(2,\F)\cong\SP(2,\F)$ (see \cite{Gro02}), the result follows immediately.
\end{proof}
\begin{theorem}\label{th:symplectic-not-acc}
The groups $\SP(2n,\F)$ are not acceptable.
\end{theorem}
\begin{proof}
    Follows easily from \cref{prop:smaller-symp} and \cref{prop:sp2-not-acc}.
\end{proof}
\begin{proposition}\label{prop:orthogonal-induction}
    The acceptability of $\Or(2m+1,\F)$ implies the acceptability of $\Or(2m-1,\F)$. 
\end{proposition}
\begin{proof}
    Fix the element 
    $a=\begin{pmatrix}
        -I_2&\\
        &&I_{2m-1}
    \end{pmatrix}\in\Or(2m+1,\F)$. Then by \cref{lem:centralizer} we get that $\ce_{\Or(2m+1,\F)}(a)$ is acceptable. Now we have that
    \begin{align*}
        \ce_{\Or(2m+1,\F)}(a)&=\left\{\begin{pmatrix}
            A&\\
            &B
        \end{pmatrix}\middle\vert \substack{A\in\GL(2,q)\\B\in\GL(2m-1,q)}\right\}\cap\Or(2m+1,\F)\\
        &\cong \left\{A\in\GL(2,\F)|AA^t=I\right\}\times\Or(2m-1,\F).
    \end{align*}
    Hence the result follows from \cref{lem:product}.
\end{proof}
\begin{proposition}\label{prop:Ortho5-not-acceptable}
    The group $\Or(5,\F)$ is not acceptable.
\end{proposition}
\begin{proof}
    We will work with the non-degenerate quadratic form of the form
    \begin{align*}
        \begin{pmatrix}
            0&0&I_2\\0&1&0\\I_2&0&0
        \end{pmatrix},
    \end{align*}
    where $I_2$ is the identity matrix of size $2\times 2$. Consider the matrices 
    \begin{align*}
        \alpha_1=\begin{pmatrix}
            1&1\\
            &1
        \end{pmatrix}\text{ and }\alpha_{a}=\begin{pmatrix}
            1&a\\&1
        \end{pmatrix},
    \end{align*}
    where $a\neq 1$. Then the matrices
    \begin{align*}
        X_1=\begin{pmatrix}
            \alpha_1&&\\&1&\\&&^t\alpha_1^{-1}
        \end{pmatrix}\text{ and }X_a=\begin{pmatrix}
            \alpha_a&&\\&1&\\&&^t\alpha_a^{-1}
        \end{pmatrix}.
    \end{align*}
    Then $X_1\cdot X_a=X_a\cdot X_1$ and both are elements of order $p$ and are conjugate in $\Or(5,\F)$. Consider the following two homomorphisms:
    \begin{align*}
        &\varphi_1,\varphi_2:\Z/p\Z\oplus\Z/p\Z\longrightarrow\Or(5,\F)\\
        \text{given by }&\varphi_1((1,0))=X_1,\varphi_1((0,1))=X_a\\
        &\varphi_2((1,0))=X_a,\varphi_2((0,1))=X_1.
    \end{align*}
    Then although these two homomorphisms are element-conjugate, they are not globally conjugate, using the same argument as of \cref{prop:gl2-not-acc}.
\end{proof}
\begin{theorem}\label{thm:OddOrth-not-acceptable}
    For $m\geq 2$ none of the groups $\Or(2m+1,\F)$ are acceptable.
\end{theorem}
\begin{proof}
    Follows from \cref{prop:Ortho5-not-acceptable} and \cref{prop:orthogonal-induction}.
\end{proof}
\begin{proposition}\label{prop:induction-orthogonal}
    The acceptability of the group $\Or(2m+2,\F)$ implies the acceptability of $\Or(2m,\F)$.
\end{proposition}
\begin{proof}
    We will work with the quadratic form 
    \begin{align*}
        \begin{pmatrix}
            0&1&&&\\
            1&0&&&\\
            &&&\ddots&&\\
            &&&&0&1\\
            &&&&1&0
        \end{pmatrix}.
    \end{align*}
    Consider the element $a=\begin{pmatrix}
        I_{2m}&\\&-I_2
    \end{pmatrix}\in\Or(2m+2,\F)$. Then by \cref{lem:centralizer}, we have that $\ce_{\Or(2m+2,\F)}(a)$ is acceptable. Note that
    \begin{align*}
        \ce_{\Or^+(2m+2,\F)(a)}&=\left\{\begin{pmatrix}
            A&\\
            &B
        \end{pmatrix}\middle\vert\substack{A\in\GL(2m,\F)\\B\in\GL(2,\F)}\right\}\cap\Or(2m+2,\F)\\
        &=\Or(2m,q)\times\Or(2,\F).
    \end{align*}
    Hence the result follows from \cref{lem:product}.
\end{proof}
\begin{proposition}\label{prop:orth+4-not-accp}
    For $q\geq 5$, the group $\Or(4,\F)$ is not acceptable.
\end{proposition}
\begin{proof}
    In this proof we will work with the quadratic form $\begin{pmatrix}
        0&I_2\\I_2&0
    \end{pmatrix}$. The proof is similar to that of \cref{prop:Ortho5-not-acceptable}, but we will replicate it with proper modification for completeness. Consider the matrices 
    \begin{align*}
        \alpha_1=\begin{pmatrix}
            1&1\\
            &1
        \end{pmatrix}\text{ and }\alpha_{a}=\begin{pmatrix}
            1&a\\&1
        \end{pmatrix},
    \end{align*}
    where $a\neq 1$. Then the matrices
    \begin{align*}
        X_1=\begin{pmatrix}
            \alpha_1&\\&^t\alpha_1^{-1}
        \end{pmatrix}\text{ and }X_a=\begin{pmatrix}
            \alpha_a&\\&^t\alpha_a^{-1}
        \end{pmatrix}.
    \end{align*}
    Then $X_1\cdot X_a=X_a\cdot X_1$ and both are element of order $p$ and are conjugate in $\Or(4,\F)$. Now consider the following two homomorphisms:
    \begin{align*}
        &\varphi_1,\varphi_2:\Z/p\Z\oplus\Z/p\Z\longrightarrow\Or(4,\F)\\
        \text{given by }&\varphi_1((1,0))=X_1,\varphi_1((0,1))=X_a\\
        &\varphi_2((1,0))=X_a,\varphi_2((0,1))=X_1.
    \end{align*}
    Then although these two homomorphisms are element-conjugate, they are not globally conjugate, using the same argument as of \cref{prop:gl2-not-acc}.
\end{proof}
\begin{theorem}\label{thm:Orth-odd-not-acc}
    For $m\geq 2$, the groups $\Or(2m,q)$ are not acceptable.
\end{theorem}
\begin{proof}
    This is an immediate consequence of \cref{prop:induction-orthogonal} and \cref{prop:orth+4-not-accp}.
\end{proof}

%% file: conc.tex
\section{Concluding remarks and future questions}\label{sec:conclusion}
\subsection{Finite groups of Lie type} Consider the finite general linear group $\GL(n,q)$, where $q$ is a power of prime. Following the proof of \autoref{thm:gl-not-acc} we easily see that these groups are not acceptable, for all $n$ and $q$. The same result holds for $\U(n,q)$ as well. 
For finite fields, in case of symplectic and orthgonal (three different) groups the unipotent conjugacy class further breaks into several conjugacy 
classes. Hence the method adopted here will not be applicable to decide the acceptability of these groups. At this point we believe that these groups are unacceptable as well. 

\subsection{Exceptional cases} We have seen that the groups discussed in this article are all 
unacceptable. On the contrary, the analogue of these groups (defined over $\mathbb C$ or 
$\mathbb R$) are acceptable (see \cite[pp. 254]{Mi94} for precise table). In the same paper 
M. Larsen proved that $F_4(\mathbb C)$ and its compact form are unacceptable (see \cite[Proposition 3.11]{Mi94}). Hence it will be interesting to see which exceptional groups (defined over 
$\overline{\mathbb F _p}$) are acceptable.
\subsection{Application to invariant theory}  Invariant theory studies the ring of those regular functions on an affine variety $X$ which
are constant on the orbits of an action of a linear algebraic group $G$ on $X$, where the
action is given by a morphism $G \times X \longrightarrow X$. Consider a field $\mathbb F$ and a subgroup $G$ of 
$\GL(V)\cong \GL(n,\mathbb F)$. Take an action of $G$ on a variety $X$. The algebra $\mathbb F[X]$ of regular functions on $X$, can be equipped with a $G$-module 
structure by defining $g\cdot f(v)=f(g^{-1}\cdot v)$ for all $g\in G$, $f\in V^*$ and $v\in V$. The ring of invariant
\begin{align*}
    \mathbb F[X]^G=\{f\in \mathbb F[X]| g\cdot f= f \text{ for all }g\in G\}
\end{align*}
is finitely generated, due to classical results of Noether (see \cite{No15}, \cite{No26}). A consequence of seminal works \cite{La18} and \cite{La18-2} of the French mathematician Vincent Lafforgue is that: when $\mathbb{F}=\mathbb{C}$ and $G$ is a connected linear reductive group and if $\mathbb{C}[G^n]^G$ (where the action is diagonal conjugation) is generated by $1$-argument invariants for all $n\geq 1$, then $G$ is acceptable. The Italian mathematician Claudio Procesi has proved that this is true for classical groups general linear groups, symplectic groups, and orthogonal groups over $\mathbb{C}$
 in his works \cite{Pr76} and \cite{Pr07}. Given that we have proved the unacceptability of $\GL(n,\F)$, $\SL(n,\F)$, $\U(n,\F)$, $\SP(2n,\F)$, $\Or^\epsilon(n,\F)$, we would like to finish by posing the question that does there exist $n\geq 1$ such that $\F[G^n]^G$ is not generated by $1$-argument invariants in case $G$  is one of the groups discussed in this paper.